\documentclass[a4paper,12pt]{amsart}

\newtheorem{theorem}{Theorem}[section]

\newtheorem{remark}{Remark}
\newtheorem{ex}{Example}
\newtheorem*{problem}{Problem}
\numberwithin{equation}{section}

\usepackage{booktabs}

\usepackage{amsmath}
\usepackage{amssymb}
\usepackage{amsfonts}
\usepackage{mathrsfs} 
\usepackage{graphicx}
\usepackage{color}
\usepackage{ulem}

\renewcommand{\emph}[1]{{\it#1}}

\newcommand{\xx}{\mathbf{x}} 
\newcommand{\yy}{\mathbf{y}} 
\newcommand{\ff}{\mathbf{f}} 
\newcommand{\cc}{c} 
\newcommand{\dd}{d} 
\newcommand{\XX}{X} 
\newcommand{\YY}{Y} 
\newcommand{\R}{\mathbb{R}} 
\newcommand{\N}{\mathbb{N}} 
\newcommand{\M}{\mathbf{M}} 
\newcommand{\K}{\mathbf{K}} 

\title{\bf Graph recovery from incomplete moment information}
\thanks{Lasserre's research is partly supported by AI Interdisciplinary Institute ANITI funding, through the French ``Investing for the Future PIA3'' program under the Grant agreement n$^{\circ}$ANR-19-PI3A-0004.}

\begin{document}

\author{Didier Henrion$^{1,2}$, Jean Bernard Lasserre$^{1,3}$}

\footnotetext[1]{CNRS; LAAS; Universit\'e de Toulouse, 7 avenue du colonel Roche, F-31400 Toulouse, France. }
\footnotetext[2]{Faculty of Electrical Engineering, Czech Technical University in Prague, Technick\'a 2, CZ-16626 Prague, Czechia.}
\footnotetext[3]{Institute of Mathematics; Universit\'e de Toulouse, 118 route de Narbonne, F-31062 Toulouse, France. }

\begin{abstract}
We investigate a class of moment problems, namely recovering a measure supported on the 
graph of a function from partial knowledge of its moments, as for instance in
some problems of optimal transport or density estimation. We show that the sole knowledge 
of first degree moments of the function, namely linear measurements, is sufficient to obtain asymptotically all the other moments
by solving a hierarchy of semidefinite relaxations (viewed as moment matrix completion problems) with a specific sparsity inducing criterion related to
a weighted $\ell_1$-norm of 
the moment sequence of the measure. The resulting
sequence of optimal solutions converges to the whole moment sequence of the measure which is shown to be the unique optimal solution of a certain infinite-dimensional linear optimization problem (LP). Then one may recover 
the function by a recent extraction algorithm based on the Christoffel-Darboux kernel associated with the measure.
Finally, the support of such a measure supported on a graph is a meager, very thin (hence sparse) set.
Therefore the LP on measures with this sparsity inducing criterion can be interpreted as  
an analogue for infinite-dimensional  signals of the LP in super-resolution for (sparse) atomic
signals.\\[1em]
\end{abstract}

\date{Draft of \today}

\maketitle

\section{Inverse problem: from moments to graph}

In data science, it is often relevant to process moments of a signal instead of the signal itself. For complex valued signals, the moments are Fourier coefficients, and many filtering operations are efficiently carried out in the sequence of moments. In numerical approximation algorithms, many operations on real valued signals are more efficiently implemented in their
sequence of Chebyshev coefficients \cite{trefethen}. In the moment-SOS hierarchy approach, many nonlinear nonconvex problems are reformulated and solved approximately in the sequence of moments; see \cite{l10,khl20} and references therein. Once moments or approximate moments have been computed, one is faced with the inverse problem of reconstructing the signal from its moments.

The recent work \cite{momgraph} describes an algorithm based on the Christoffel-Darboux kernel, to recover the graph of a function from knowledge of its moments. Its novelty (and distinguishing feature) 
is to approximate the \emph{graph} of the 
function (rather than the function itself) with a \emph{semialgebraic} function (namely a minimizer of a sum of squares of polynomials) with $L_1$ and pointwise convergence guarantees for an increasing number of input moments. In contrast to approximations via continuous functions (e.g. polynomials) 
this larger class of approximants permits to better handle nasty functions (e.g. with discontinuities). In particular, this distinguishing 
feature allows to sometimes avoid a typical Gibbs phenomenon (as well as oscillations) encountered when approximating a discontinuous function by polynomials;
for more details and illustrative examples the interested reader is referred to \cite{momgraph}.

Let $\N:=\{0,1,2,\ldots\}$ be the set of natural numbers.
The (algebraic) moment of degree $\dd:=(\dd_x,\dd_y)\in \N^{n_x}\times\N^{n_y}$ of the Young measure supported on the graph of a given mapping  $\ff : \XX \to \YY$ from a domain $\XX \subset \R^{n_x}$  to a domain $\YY \subset \R^{n_y}$ is the real number
\begin{equation}
    \label{aux1}
\int_\XX \xx^{\dd_x} \ff(\xx)^{\dd_y}\,d\lambda(\xx)\,,\quad (d_x,d_y)\in\N^{n_x}\times\N^{n_y}\,,
\end{equation}
where we have used the multi-index monomial notation\footnote{For numerical reasons, it is preferable to express
moments in a basis different from the monomial basis. However, in this note we stick to monomials for notational ease.}. In \eqref{aux1} $\lambda$ is a given reference measure, for example the uniform measure on $\XX$.

In applications treated in \cite{momgraph} to recover $\ff$, all moments \eqref{aux1} 
are available; more precisely, in principle an arbitrarily large number of them can be computed. However in a number of other interesting applications
this is not the case. For instance:
\begin{itemize}
\item In the moment-SOS approach for solving or controlling differential equations, only finitely many (approximate) moments are given, i.e. $|d|\leq r$ where $r \in \N$ is a given relaxation order, and $|d|$ denotes the sum of the entries of vector $d$;
\item In statistics for density estimation ($n_y=1$), only finitely many moments 
$\int \xx^{d_x} f(\xx)\,d\xx$, $d_x\in \N^{n_x}$,
of the unknown density function $f\geq 0$ are available, i.e. $\dd_y = 1$;
\item In optimal transport, only marginal moments are available, i.e. either $|\dd_x|=0$ or $|\dd_y|=0$.
\end{itemize}

The main feature of this paper is to consider the more general inverse problem of recovering the graph of a function $\ff$ given some of its moments, indexed by $\dd \in D$ in a given countable index set $D \subset \N^{n_x}\times\N^{n_y}$. The following \emph{optimization / relaxation / concentration} features combined with the recovery technique described in \cite{momgraph},
are key ingredients for a successful recovery of $\ff$.

\subsection*{Optimization}
In order to approximate the function $\ff$ from some
of its moments \eqref{aux1}, an integral functional
\[
\int_\XX \cc(\xx,\ff(\xx))\,d\lambda(\xx)
\]
of the solution is typically minimized, where $\cc : \XX \times \YY \to \R$ is a running cost. In some applications, the cost is given and there is a unique function solving the problem.  Otherwise, the cost is chosen so that a specific function is recovered.
The resulting optimization problem reads
\begin{equation}\label{opt}
\inf_{\ff} \int_\XX \cc(\xx,\ff(\xx))\,d\lambda(\xx)
\quad\mathrm{s.t.} \quad \int_\XX \xx^{\dd_x} \ff(\xx)^{\dd_y}\,d\lambda(\xx) = z_d, \quad d \in D
\end{equation}
where the sequence $(z_d)_{d\in D} \subset \R$ and the index set $D$ are given. Function $\ff$ is sought in a given functional space, for example square integrable functions, uniformly bounded functions, or measurable functions.

\subsection*{Relaxation} 
Optimization problem \eqref{opt} is nonlinear, nonconvex and quite difficult in general. Following Kantorovich's idea proposed for transport problems \cite[Introduction to optimal transport]{santambrogio}, one can instead consider the 
relaxed version
\begin{equation}\label{relax}
\inf_{\mu_\xx} \int_{\XX\times\YY} \cc(\xx,\yy)\mu_\xx(d\yy)\lambda(d\xx)
\quad\mathrm{s.t.} \quad \int_{\XX\times\YY} \xx^{\dd_x} \yy^{\dd_y}\mu_\xx(d\yy)\lambda(d\xx) = z_d, \quad d \in D
\end{equation}
of the non convex problem \eqref{opt}. A distinguishing feature of the relaxation \eqref{relax} is to be \emph{linear} and hence convex in $\mu_\xx$, which is a probability measure on $\YY$ parametrized by $\xx \in \XX$, also called a conditional, or Young measure. Weak star compactness arguments can be used to show that under mild assumptions on the input data, the infimum in relaxed problem \eqref{relax} is attained. Letting $\mu_\xx=\delta_{\ff(\xx)}$ shows that the infimum in relaxed problem \eqref{relax} is less than or equal to the infimum in problem \eqref{opt}. Techniques of calculus of variations can be used to show that under mild assumptions on the input data, the minimum in relaxed problem \eqref{relax} is equal to the infimum in problem \eqref{opt}, i.e. that there is no relaxation gap.

\subsection*{Concentration}
For some specific choices of $c$ and $D$, it can be proved that $\mu_\xx=\delta_{\ff(\xx)}$ 
is the unique optimal solution of relaxed problem \eqref{relax},
and hence problem \eqref{opt} has a unique solution $\ff$. In this case we say that the Young measure $\mu_\xx$ is concentrated on the graph of function $\ff$.

For instance this happens in quadratic optimal transport for which the cost $c(\xx,\yy):=\frac{1}{2}\|\xx-\yy\|^2$ is given. Brenier showed that the optimal transport function is the gradient of a convex function; see e.g. \cite[Chapter 1]{santambrogio} and references therein. If there is freedom for the cost $c$, it can be chosen so that concentration occurs.

It also happens in the moment-SOS approach for solving a certain class of nonlinear partial differential equations, if additional linear constraints are enforced on the measure and/or the cost if chosen appropriately; see e.g. the recent  work on Burgers equation \cite{burgers}.

\subsection*{Contribution}
We consider problem \eqref{relax} 
in the case $|d_y|\leq 1$, i.e. the only available moments are those of the reference measure (zero-th order moments $d_y=0$) and averages of the graph (first order moments $d_y=1$, namely linear measurements of the function). A typical case is density estimation.

Our contribution is twofold: 
I. As for transport problem,
we show that the measure $\mu$ to recover is the unique optimal solution of a certain infinite-dimensional LP on a measure space. II. Next, under the assumption that the input data are semialgebraic, we provide a systematic numerical scheme to approximate as closely as desired any arbitrary finite number of moments of $\mu$. It consists of solving problem \eqref{relax} with the moment-SOS hierarchy. Since we are given an incomplete set of moments, we 
optimize over the remaining set of unknown moments. In other words, each semidefinite relaxation
is a \emph{moment matrix completion problem} of increasing size.

The rationale behind the asymptotic convergence of the hierarchy is that every semidefinite relaxation is 
a $r$-truncated version of the infinite-dimensional LP where one only 
considers pseudo-moments of order $r$. It turns out that as $r$ increases, the resulting sequence 
of optimal solutions converges to the vector of moments of $\mu$.

Interestingly this result can be interpreted as an 
analogue for infinite-dimensional signals (densities) of \emph{super-resolution} for sparse atomic measures \cite{dCG12,CF14,dC17}. Indeed
instead of considering the measure $d\nu(\xx)=\ff(\xx)\,d\xx$ on $\XX$ with unknown density 
$\ff$ (but with \emph{all} its moments available), we rather consider 
the graph density measure $\mu=\delta_{\ff(\xx)}(dy)\lambda(d\xx)$ whose support $(\xx,\ff(\xx))$ is a meager (very thin) degenerate subset of $\XX\times\YY$ (hence sparse), now with partial knowledge of its moments.
Then as for super-resolution of sparse atomic measures, this sparse (degenerate) measure is the unique solution of an infinite LP on measures where the sparsity-inducing cost to minimize is a weighted $\ell_1$-norm of their moment vector. 

Then once moments of $\mu$ up to order $2r$ have been 
approximated in an optimal solution at step $r$ of the hierarchy,
the function can  be in turn approximated 
with the reconstruction technique already alluded to and described in 
\cite{momgraph}; the deeper in the hierarchy the better is the resulting approximation. 

Our experiments convey the following message: In the above situations it is better to approximate the graph $G:=\{(\xx,f(\xx): \xx\in X)\}$ from moments of the measure $\mu$ supported on $G$ (a thin subset of $X\times Y$), rather than approximate the function  itself (e.g. $L^2$-norm approximation by polynomials) from moments
of the measure $d\nu=fd\xx$ on $X$ with density $f$. Even more, this is true even if most moments of $\mu$ have to be approximated,
whereas all exact moments of $\nu$ are available.

\section{Problem statement}
Let $\R[\xx]$ denote the ring of real-valued polynomials in the variables $\xx=(x_1,\ldots,x_n)$
and $\Sigma[\xx]\subset\R[\xx]$ its subset of sums-of-squares (in short SOS) polynomials.
Let $\XX\,:=\,\{\xx\in\R^n:\: g_j(\xx)\,\geq\,0,\:j=1,\ldots,m\,\}$
be a compact basic semi-algebraic set defined
for given polynomials $(g_j)\subset\R[\xx]$. 
With no loss of generality, we will assume that $\XX\subseteq[0,1]^n$ (possibly after some scaling and translation).

The convex cone of Borel regular non-negative measures on $\XX$, topologically dual to the convex cone of non-negative continuous functions on $\XX$, is denoted by $\mathscr{M}(\XX)_+$. The Borel sigma-algebra of subsets of $\XX$ is denoted by $\mathscr{B}(\XX)$.

Let $\lambda\in\mathscr{M}(\XX)_+$ be a finite probability measure with
moments $\boldsymbol{\lambda}=(\lambda_{d_x}) \subset \R$ defined by
\[
\lambda_{d_x}:=\int_X \xx^{d_x}d\lambda(\xx), \quad d_x \in \N^n.
\]
From now on, let us restrict our attention to scalar-valued functions $f:\XX\to\YY$. Vector-valued functions can be treated entry-wise.
Let $\mu\in\mathscr{M}(\XX\times\YY)_+$ be a finite probability measure of the form
 $d\mu(\xx,y)=\delta_{f(\xx)}(dy)\lambda(d\xx)$, i.e.
  \begin{equation}
     \label{measure:mu}
 \mu(A\times B)\,=\,\int_A\,1_B(f(\xx))\,d\lambda(\xx),\quad\forall A\in\mathscr{B}(\XX),\,B\in\mathscr{B}(\YY),
 \end{equation}
 with all finite moments 
 $\boldsymbol{\mu}=(\mu_d)_{d\in\mathbb{N}^{n+1}} \subset \R$ defined by
\[\mu_d\,:=\,\int_\XX \int_\YY \xx^{d_x}\,y^{d_y}\,d\mu(\xx,y)\,=\,
 \int_{\XX}\xx^{d_x}\,f(\xx)^{d_y}\,d\lambda(\xx), \quad d=(d_x,\:d_y)\in\N^{n+1}\,.\]
In particular $\mu_{d_x,0} = \lambda_{d_x}$ for all $d_x\in\N^n$.
We assume that $0\leq f\in \mathscr{L}^\infty(\lambda)$ and $\Vert f\Vert_\infty\leq \gamma$ for a given $\gamma>0$, i.e.
$\YY:=[0,\gamma]$. (The case where $f$ is bounded but can take negative values 
reduces to the previous case by taking $\tilde{f}:=f+\Vert f\Vert_\infty$.)
Let $\N^{n+1}_r:=\{d \in \N^{n+1} : |d|:=\sum_{k=1}^{n+1} d_k \leq r\}$.

\begin{problem}\label{pb-completion}{\bf Moment completion.}
Let $D:=\{d=(d_x,\,d_y) \in \N^{n+1} : d_y \leq 1\}$.
Given moments $(\mu_d)_{d \in D}$ and given $s\in\N$, approximate as closely as desired all moments $(\mu_d)_{d \in \N^{n+1}_{2s}}$.
\end{problem}

\subsection*{Further notations}~
Let $g_0(\xx):=1$, $g_{m+1}(\xx):=n-\Vert\xx\Vert^2$ and $g_{m+2}(y):=y(\gamma-y)$. As $g_{m+1}(\xx)\geq0$ defines a redundant constraint (since $X\subseteq [0,1]^n$), it can always be included 
in the definition of $\XX$.  
Then the quadratic module 
\begin{equation}    \label{Q(g)}
Q(g)\,:=\,\{\,\sum_{j=0}^{m+2}\sigma_j\,g_j\::\:\sigma_j\in\Sigma[\xx,y]\,\}\subset\R[\xx,y]\,,
\end{equation}
generated by the $g_j$ is Archimedean, where $\Sigma[\xx,y]$ denotes the convex cone of SOS polynomials in $\R[\xx,y]$.

For $j=0,1,\ldots,m+2$, let $d_j:=\lceil{\rm deg}(g_j)/2\rceil$, and express $g_j$ in the monomial basis as $g_j(\xx,y) = \sum_{d=(d_x,d_y) \in \N^{n+1}} {g_j}_d\, \xx^{d_x}y^{d_y}$. 

Next, define the \emph{localizing matrix} $\M_{s-d_j}(g_j\,\boldsymbol{\mu}))$ 
associated with $\boldsymbol{\mu}$ and $g_j$, as the real symmetric matrix of size $\binom{n+1+s-d_j}{n+1}$ whose entries indexed by $(d_r,d_c) \in\mathbb{N}^{n+1}_{s-d_j}\times\mathbb{N}^{n+1}_{s-d_j}$ are equal to $\sum_{d \in \N^{n+1}} {g_j}_d\, \mu_{d+d_r+d_c}$. In particular, for $j=0$, $g_0=1$, $d_0=0$ and $\M_s(\boldsymbol{\mu})$ is called the \emph{moment matrix} associated with $\boldsymbol{\mu}$.

\subsection*{A preliminary result}
The following result of independent interest provides a \emph{rationale} for the 
infinite-dimensional LP introduced in Section \ref{LP} as well as for the numerical scheme described  in Section \ref{sec:main} to recover the measure $\mu$ in \eqref{measure:mu} supported on the graph of $f$.

\begin{theorem}
\label{th1}
Let $\mu\in\mathscr{M}(\XX\times\YY)_+$ be as in \eqref{measure:mu} and let $\phi\in\mathscr{M}(\XX\times\YY)_+$ be such that all its moments 
$\boldsymbol{\phi}=(\phi_d)_{d\in\N^{n+1}}$ are finite, with
$\phi_{d_x,0}=\lambda_{d_x}$ and $\phi_{d_x,1}=\mu_{d_x,1}$ for all $d_x\in\N^n$. Then $\phi_d\geq\mu_d$ for all $d\in\N^{n+1}$. In particular,
\begin{equation}
\label{eq:trace}
    {\rm trace}\:\M_r(\boldsymbol{\mu})\,\leq\,{\rm trace}\:\M_r(\boldsymbol{\phi}),\quad\forall r\in\N.
\end{equation}
\end{theorem}

The proof of Theorem \ref{th1}  is postponed to Section \ref{proofs}.

\begin{remark}
\label{positive-0}
If $X\not\subset\R^n_+$ then we obtain the slightly weaker result $\phi_{2d_x,d_y}\geq\mu_{2d_x,d_y}$ for all 
$(2d_x,d_y)\in\N^{n+1}$, and therefore 
\eqref{eq:trace} still holds. 
\end{remark}

\section{An infinite-dimensional LP on measures}
\label{LP}
The moment-completion problem can be viewed as a way to embed the initial problem
into that of recovering 
a curve (more exactly the measure $\mu$ supported on the graph of that curve), i.e.,
a lower dimensional object with a meager, very thin, degenerate support 
in a space of measures on $\XX\times Y$ which can have support of full-dimension $n+1$. 
In other words this meager support 
$\{(\xx,f(\xx)): \xx\in\XX\}$ is sparse in $\XX\times Y$.

Therefore when viewing the problem 
with such glasses, Theorem \ref{th1} suggests a way to build up
an appropriate cost criterion defined over $\mathscr{M}(\XX\times Y)_+$ that
will select $\mu$ in \eqref{measure:mu} as the \emph{unique} optimal solution of a certain infinite-dimensional LP on $\mathscr{M}(\YY\times Y)_+$.

So let $\boldsymbol{\theta}=(\theta_d)_{d\in\N^{n+1}}$ be a nonnegative sequence 
of $\ell_1$, the Banach space of summable sequences $\mathbf{a}=(a_d)$ with norm $\Vert \mathbf{a}\Vert_1=\sum_{d\in\N^{n+1}}\vert a_d\vert$.
For instance choose
\begin{equation}
    \label{eq:theta}
\theta_{d=(d_x,d_y)}\,=\,
\left(c^{2d_y+1}\,{n-1+\vert d_x\vert\choose \vert d_x\vert}\right)^{-1}\quad \mbox{with $c>\max(1,\gamma)$.}\end{equation}
Recalling $Y=[0,\gamma]$, observe that $\phi_{d_x,d_y}=\displaystyle\int_{\XX\times Y}\xx^{d_x}\,y^{d_y}\,d\phi\,<\,\gamma^{d_y}$ for all $(d_x,d_y)\in\N^{n+1}$, and therefore
\begin{equation}    \label{def-T}
   {\mathscr T}(\boldsymbol{\phi})\,:=\,\sum_d\theta_d\,\phi_d\,<\,
\sum_{\vert d_x\vert=0}^\infty\sum_{d_y} \frac{\gamma^{d_y}}{c^{2d_y+1}\,{n-1+\vert d_x\vert\choose \vert d_x\vert}}\,<\,\frac{1}{c}\,
\sum_{d_y} (\gamma/c^2)^{d_y}\,<\,\frac{1}{c-\gamma/c}\,,\end{equation}
for all $\phi\in\mathscr{M}(\XX\times Y)_+$.
Next, with $\theta$ as in \eqref{eq:theta} let $\mathrm{T}:\XX\times\YY\to\R_+$,
\[(\xx,y)\,\mapsto\, \mathrm{T}(\xx,y)\,:=\,\sum_{d=(d_x,d_y)\in\N^{n+1}}\theta_d \,\xx^{d_x}\,y^{d_y} \,<\,\infty.\]
Observe that for all $(\xx,y)\in\XX\times\YY$,
\[0\,\leq\,\sum_{d=(dx,dy)\leq r}\theta_d\,\xx^{d_x}\,y^{d_y}
\,=:\,\mathrm{T}_r(\xx,y)\,\uparrow\,\mathrm{T}(\xx,y)\quad\mbox{as $r\to\infty$}.\]
Hence by Monotone Convergence \cite[p.44]{ash},
\begin{equation}
    \label{monotone}
\mathscr{T}(\boldsymbol{\phi})=\sum_{d\in\N^{n+1}}\theta_d\,\phi_d\,=\,\lim_{r\to\infty}\sum_{\vert d\vert\leq r}\theta_d\,\phi_d\,=\,
\lim_{r\to\infty} \int\mathrm{T}_r(\xx,y)\,d\phi\,=\,
\int\mathrm{T}(\xx,y)\,d\phi,\end{equation}
for all $\phi\in\mathscr{M}(\XX\times Y)_+$.

\subsection{An LP on measures}

Consider the infinite-dimensional LP:
\begin{equation}
\label{relax-primal-infty}
    \begin{array}{rl}
    \tau'_\infty=\displaystyle\inf_{\phi\in\mathscr{M}(\XX\times\YY)_+} & \displaystyle\int \mathrm{T}\,d\phi\\
    \mbox{s.t.} 
    &\displaystyle\int \xx^{d_x}\,y^{d_y}\,d\phi \, =\,
    \mu_d,\quad d=(d_x,d_y) \in D\,.
        \end{array}
\end{equation}
\begin{theorem}
\label{th-final}
The measure $\mu=\delta_{f(\xx)}(dy)\,\lambda(d\xx)\in\mathscr{M}(\XX\times\YY)_+$ is the unique optimal solution of the infinite-dimensional LP \eqref{relax-primal-infty}.
\end{theorem}
\begin{proof}
Let $\phi\in\mathscr{M}(\XX\times\YY)_+$ be an arbitrary  feasible solution of \eqref{relax-primal-infty}, with moments
$(\phi_d)_{d\in\N^{n+1}}$. Then:
\begin{eqnarray*}
\displaystyle\int \mathrm{T}(\xx,y)\,d\phi
&=&\sum_{d=(d_x,d_y)\in\N^{n+1}}\theta_d \,\phi_d\quad\mbox{[by \eqref{monotone}]}\\
&\geq&\sum_{d=(d_x,d_y)\in\N^{n+1}}\theta_d \,\mu_d\quad\mbox{[by Theorem \ref{th1}]}\\
&=&\displaystyle\int \mathrm{T}(\xx,y)\,d\mu\,,\quad\mbox{[by \eqref{monotone}]},
\end{eqnarray*}
which implies that  $\mu$ is an optimal solution of LP \eqref{relax-primal-infty}. As measures on compact sets are moment determinate, uniqueness follows.
\end{proof}
Notice that the criterion $\int\mathrm{T}\,d\phi=\mathscr{T}(\boldsymbol{\phi})$ (a weighted $\ell_1$-norm of the moment vector of $\phi$) is \emph{sparsity-inducing}. Indeed it induces a sparse support for the unique optimal
solution of LP \eqref{relax-primal-infty}. Again, by sparse we mean 
that the support of $\mu$ is degenerate and a meager (thin) set of $\XX\times\YY$ (since it is the graph $(\xx,f(\xx))$).

\begin{remark}
A parallel can be drawn with super-resolution in signal processing \cite{dCG12,CF14,dC17}, which consists of recovering an atomic measure with sparse finite support (a few atoms), from the sole knowledge of only a limited number of its moments.
Indeed as shown in \cite{CF14}, the sparse atomic measure (signal) 
to recover is the unique optimal solution of an LP on measures with (spasity-inducing) total-variation criterion. In addition, solving the latter LP reduces to solving a single SDP in the univariate case.

In our graph-recovering inverse problem, our knowledge is also limited to
a few moments only, namely $(\mu_{d,0}, \mu_{d,1})_{d\in\N^n}$ out
of the potentially many $(\mu_{d_x,d_y})_{(d_x,d_y)\in\N^{n+1}}$. (As the object to recover is infinite-dimensional
one cannot expect exact recovery from finitely many moments.)
Hence the spirit of the infinite-dimensional LP \eqref{relax-primal-infty} 
is that a measure with sparse support (the graph of a function)
can be recovered from a few moments only, as the unique optimal solution of an LP on measures with appropriate sparsity-inducing criterion.
\end{remark}

\subsection{A dual of \eqref{relax-primal-infty}}

Let $\mathscr{C}(\XX\times Y)$ be the space of continuous functions on $\XX\times Y$.
Observe that $\mathrm{T}\in\mathscr{C}(\XX\times Y)$. Indeed
let $(\xx_n,y_n)_{n\in\N}\subset \XX\times Y$ be such that
$(\xx_n,y_n)\to (\xx,y)\in \XX\times Y$, as $n\to\infty$.
Let $a^n_d:=\xx^{d_x}_n\,y^{d_y}_n$ for all $n$ and all $d\in\N^{n+1}$,
so that for all $d$, $a^n_d\to a_d:=\xx^{d_x}\,y^{d_y}$ as $n\to\infty$.
Moreover $\vert a^n_d\vert\leq y^{d_y}$ for all $n$, and 
$\sum_d \theta_d\,y^{d_y}<\infty$. Therefore 
by Dominated Convergence \cite[p. 49]{ash}:
\[\mathrm{T}(\xx_n,y_n)\,=\,\sum_d \theta_d \,\xx^{d_x}_n\,y^{d_y}_n\,\to
\sum_d \theta_d\, \xx^{d_x}\,y^{d_y}\,=\,\mathrm{T}(\xx,y)\,\quad\mbox{as $n\to\infty$}.\]
Next, consider the infinite-dimensional LP problem:
\begin{equation}
\label{relax-primal-infty-dual}
\begin{array}{rl}
\tau^*_\infty=\displaystyle\sup_{h_1,h_2\in\mathscr{C}(\XX)}&\displaystyle\int_{\XX} (h_1(\xx) +h_2(\xx)\,f(\xx)) \,d\xx\\
&h_1(\xx)+h_2(\xx)\,y\,\leq\,\mathrm{T}(\xx,y)\,,\quad \forall (\xx,y)\in\XX\times Y\,.\end{array}
\end{equation}

\begin{theorem}
\label{duality-gap}
There is no duality gap between \eqref{relax-primal-infty} and its dual \eqref{relax-primal-infty-dual}.
\end{theorem}
\begin{proof}
The LP \eqref{relax-primal-infty} reads:
\[\tau'_\infty=\displaystyle\inf_{\phi\in\mathscr{M}(\XX\times Y)_+} \{\,\langle \,\mathrm{T},\phi\,\rangle: \mathbf{A}_1\,\phi= \lambda\,;\:
\mathbf{A}_2\,\phi= \mathbf{A}_2\mu\,\},\]
where $\mathbf{A}_i:\mathscr{M}(\XX\times Y)\to\mathscr{M}(\XX)$, $i=1,2$,  are defined by:
\begin{eqnarray*}
\phi\mapsto \mathbf{A}_1\,\phi(B)&=&\int_{B \times Y}d\phi\quad B\in\mathcal{B}(\XX)\\
\phi\mapsto \mathbf{A}_2\,\phi(B)&=&\int _{B\times Y}y\,d\phi(\xx,y),\quad B\in\mathcal{B}(\XX).\end{eqnarray*}
The adjoints $\mathbf{A}_i^*:
\mathscr{C}(\XX)\to\mathscr{C}(\XX\times Y)$ are defined by:
\[h\mapsto [\mathbf{A}_1^*\,h](\xx,y)\,=\,h(\xx)\,;\quad
h\mapsto [\mathbf{A}_2^*\,h](\xx,y)\,=\,
y\,h(\xx)\,,\quad \forall (\xx,y)\in\XX\times Y,\]
for all $h\in\mathscr{C}(\XX)$.
The dual \eqref{relax-primal-infty-dual} reads:
\[\begin{array}{rl}
\tau^*_\infty=\displaystyle\sup_{h_1,h_2\in \mathscr{C}(\XX)} &
\{\langle h_1,\lambda\rangle+\langle h_2,\mathbf{A}_2\,\mu\rangle:\\ &\mathrm{T}-\mathbf{A}^*_1\,h_1{\color{red}-}\mathbf{A}^*_2\,h_2\in \mathscr{C}(\XX\times Y)_+\,\}
\end{array}\]
Equivalently:
\[\begin{array}{rl}
\tau^*_\infty=\displaystyle\sup_{h_1,h_2\in \mathscr{C}(\XX)} &
\{\displaystyle\int_{\XX}(h_1(\xx)+h_2(\xx) \,f(\xx))\,d\xx\,:\\ 
& h_1(\xx)+y\,h_2(\xx)\,\leq\,\mathrm{T}(\xx,y)\,,\quad\forall (\xx,y)
\in \XX\times Y\}.\end{array}\]
To prove that $\tau'_\infty=\tau^*_\infty$ it suffices to prove that
the set
\[\mathcal{M}:=\{(\mathbf{A}_1\,\phi,\mathbf{A}_2\,\phi,\langle \mathrm{T},\phi\rangle):\:\phi\in\mathscr{M}(\XX\times Y)_+\}\]
is closed in the weak-star topology
$\sigma(\mathscr{M}(\XX)^2\times \R,\mathscr{C}(\XX)^2\times \R$), see \cite[Theorem IV.7.2]{barvinok}.

So let $(\mathbf{A}_1\,\phi_n,\mathbf{A}_2\,\phi_n,\langle \mathrm{T},\phi_n\rangle)\to(a,b,c)$  as $n\to\infty$. From the definition of $\mathbf{A}_1$ and the weak-star convergence, one obtains 
$\phi_n(\XX\times Y)\to a(\XX\times Y)\geq0$. Hence the sequence $(\phi_n)_{n\in\N}$ is norm-bounded. As $\XX\times Y$ is compact there is a measure 
$\psi\in\mathscr{M}(\XX\times Y)_+$ and a subsequence $(n_k)_{k\in\N}$ such that
$\phi_{n_k}\to\psi$ (weak-star) as $k\to\infty$. But then
\[\langle h,a\rangle\,=\,\lim_{k\to\infty}\,\langle h,\mathbf{A}_1\,\phi_{n_k}\rangle\,=\,
\lim_{k\to\infty}\,\langle \underbrace{\mathbf{A}_1^*\,h}_{\in\mathscr{C}(\XX\times Y)},\phi_{n_k}\rangle\,=\,\langle \mathbf{A}_1^*\,h,\psi\rangle=\langle h,\mathbf{A}_1\,\psi\rangle,\quad\forall h\in\mathscr{C}(\XX),\]
which implies $a=\mathbf{A}_1\,\psi$. Similarly,
\[\langle h,b\rangle\,=\,\lim_{k\to\infty}\,\langle h,\mathbf{A}_2\,\phi_{n_k}\rangle\,=\,
\lim_{k\to\infty}\,\langle \underbrace{\mathbf{A}_2^*\,h}_{\in\mathscr{C}(\XX\times Y)},\phi_{n_k}\rangle\,=\,\langle \mathbf{A}_2^*\,h,\psi\rangle=\langle h,\mathbf{A}_2\,\psi\rangle,\quad\forall h\in\mathscr{C}(\XX),\]
which implies $b=\mathbf{A}_2\,\psi$. Finally, 
as $\mathrm{T}\in\mathscr{C}(X\times Y)$, $c=\lim_{k\to\infty}\,\langle \mathrm{T},\phi_{n_k}\rangle=
\langle \mathrm{T},\psi\rangle$
and therefore $(a,b,c)=(\mathbf{A}_1\,\psi,\mathbf{A}_2\,\psi,\langle \mathrm{T},\psi\rangle)$ for some $\psi\in\mathscr{M}(\XX\times Y)_+$, the desired result.
\end{proof}
So as for super-resolution for atomic measure
where the dual problem provides a certificate, the dual 
\eqref{relax-primal-infty-dual} indicates what would be a certificate
if it has an optimal solution.

Indeed assume that $h^*_1,h^*_2\in\mathscr{C}(\XX)$ is an optimal solution
of \eqref{relax-primal-infty-dual}. Then
\[\int_\XX (\mathrm{T}(\xx,f(\xx))-h^*_1(\xx)-h^*_2(\xx)\,f(\xx))\,d\lambda(\xx)\,=\,0.\]
Equivalently:
\begin{eqnarray*}
\mathrm{T}(\xx,f(\xx))&=&h^*_1(\xx)+h^*_2(\xx)\,f(\xx),\quad \mbox{a.e. in $\XX$.}\\
\mathrm{T}(\xx,y)&\geq&h^*_1(\xx)+h^*_2(\xx)\,y,\quad \forall (\xx,y)\in\XX\times Y.
\end{eqnarray*}
\begin{ex}
$\XX=[0,1]\,;\: Y=[0,1]$ and $f(\xx)=1$. We can take
\[\mathrm{T}(\xx,y)\,=\,\sum_{k=0}^\infty \sum_{j=0}^\infty \frac{y^j}{j{\rm !}}\frac{\xx^k}{k{\rm !}}\,=\,
\exp(\xx)\,\exp(y)\,,\quad (\xx,y)\in [0,1]\times [0,1].\]
Then with $h^*_1=0$, $h^*_2:=e\cdot\exp(\xx)$ we obtain
\[\mathrm{T}(\xx,f(\xx))=e\cdot\exp(\xx)\,=\,f(\xx)\cdot h^*_2(\xx)\,;\quad 
\mathrm{T}(\xx,y)\,=\,\exp(\xx)\cdot\exp(y)\,\geq\,e\cdot\exp(\xx)\,y\]
as $\exp(y)\geq e\,y$ for all $y\in [0,1]$.
\end{ex}
Unfortunately solving the infinite-dimensional LP \eqref{relax-primal-infty} is 
out of reach. However Theorem \ref{th1} still suggests a practical numerical scheme
which in principle allow to approximate as closely as desired moments, any finite number of moments of $\mu$ in \eqref{measure:mu}.

\section{Asymptotic Recovery}
\label{sec:main}
In this section we provide a systematic numerical scheme to recover
asymptotically either an arbitrary (but fixed) number of moments of $\mu$, or
all moments of $\mu$.

\subsection{Recovery of finitely many moments}
Given an integer $s\in\N$ fixed, and a sequence 
$\boldsymbol{\phi}=(\phi_d)_{d\in\N^{n+1}}$, let 
\[{\mathscr T}_s(\boldsymbol{\phi})\,:=\,\sum_{d\in\N^{n+1}_{s}}\phi_{2d}\,=\,{\rm trace}(\M_s(\boldsymbol{\phi})),\]
and consider the following hierarchy of semidefinite programs:
\begin{equation}
\label{relax-primal}
    \begin{array}{rl}
    \tau_r=\displaystyle\inf_{\boldsymbol{\phi}}&
    {\mathscr T}_s(\boldsymbol{\phi})\\
    \mbox{s.t.} 
    &\M_r(\boldsymbol{\phi})\,\succeq0\,,\quad \M_{r-d_j}(g_j\,\boldsymbol{\phi})\succeq0,\quad j=1,\ldots,m+2,\\
    &\phi_d \, =\,\mu_d,\quad d \in D,
        \end{array}
\end{equation}
indexed by the relaxation order $r\geq d_0=\max(s,\max_j d_j)$, and where 
$\boldsymbol{\phi}=(\phi_d)_{d\in\N^{n+1}_{2r}}$.
\begin{theorem}
\label{th-finite}
The semidefinite program \eqref{relax-primal} has an optimal solution $\boldsymbol{\phi}^r=(\phi^r_d)_{d\in\N^{n+1}_{2r}}$, and  $\tau_r={\mathscr T}_s(\boldsymbol{\phi}^r)\leq {\mathscr T}_s(\boldsymbol{\mu})$ for all $r \geq d_0$. 
In addition, $\lim_{r\to\infty} {\mathscr T}_s(\boldsymbol{\phi}^r)={\mathscr T}_s(\boldsymbol{\mu})$ and
\begin{equation}
    \label{th-finite-completion}
    \lim_{r\to\infty}\phi^r_d\,=\,\mu_d\quad\forall d\in\N^{n+1}_{2s}.
\end{equation}
Moreover if $\nu$ is a measure on $\XX\times \YY$ such that $\nu_d=\mu_d$ for all $d\in D$ then
${\mathscr T}_s(\boldsymbol{\mu})\leq {\mathscr T}_s(\boldsymbol{\nu})$.

\end{theorem}

In other words, with a suitable choice of an objective function, 
the moment completion problem is solved asymptotically via the hierarchy of semidefinite relaxations \eqref{relax-primal}. That is,
arbitrary finitely many moments of a graph can be recovered
from  linear measurements, i.e. from the sole knowledge of zero-th and first order moments.

The proof of Theorem \ref{th-finite} is relegated to Section \ref{proofs}.

\subsection{Recovery of all moments}

Theorem \ref{th-finite} states that finitely many moments of the graph can be recovered. This result can be extended to recover (asymptotically) infinitely many moments.
Recall the weighted criterion:
\[{\mathscr T}(\boldsymbol{\phi})\,:=\,
\sum_{d\in\N^{n+1}}\theta_d\,\phi_d\]
already considered in \eqref{def-T}. When $\boldsymbol{\phi}=(\phi_d)_{d\in\N^{n+1}_{2r}}$ then ${\mathscr T}(\boldsymbol{\phi})$ is the obvious finite truncation
\[{\mathscr T}(\boldsymbol{\phi})\,=\,\sum_{d\in\N^{n+1}_{2r}}\theta_d\,\phi_d.\]
Next, consider the following hierarchy of semidefinite programs:
\begin{equation}
\label{relax-primal-new}
    \begin{array}{rl}
    \tau'_r=\displaystyle\inf_{\boldsymbol{\phi}} & {\mathscr T}(\boldsymbol{\phi})\::\\
    \mbox{s.t.} 
    &\M_r(\boldsymbol{\phi})\,\succeq0\,,\quad \M_{r-d_j}(g_j\,\boldsymbol{\phi})\succeq0,\quad j=1,\ldots,m+2,\\
    &\phi_d \, =\,\mu_d,\quad d \in D,
        \end{array}
\end{equation}
indexed by $r\geq d_0:=\max_j d_j$. Notice that  in contrast to \eqref{relax-primal}, the cost function
in \eqref{relax-primal-new} changes with the index $r$ as more and more terms 
of $\boldsymbol{\phi}$ are taken into account. 
Of course as $\boldsymbol{\mu}$ is a feasible solution, we also have $\tau'_d\leq {\mathscr T}(\boldsymbol{\mu})$ for every $r$.

\begin{theorem}
\label{th-infinite}
The semidefinite program \eqref{relax-primal-new} has an optimal solution $\boldsymbol{\phi}^r=(\phi^r_d)_{d\in\N^{n+1}_{2r}}$, and
$\tau'_r={\mathscr T}(\boldsymbol{\phi}^r)\leq{\mathscr T}(\boldsymbol{\mu})$ for all $r\geq d_0$. In addition, $\lim_{r\to\infty}  
{\mathscr T}(\boldsymbol{\phi}^r)={\mathscr T}(\boldsymbol{\mu})$ and
\begin{equation}
    \label{th-infinite-completion}
    \lim_{r\to\infty}\phi^r_d\,=\,\mu_d\quad\forall d\in\N^{n+1}.
\end{equation}
\end{theorem}

The proof of Theorem \ref{th-infinite} is relegated to Section \ref{proofs}.

\section{Application to the L-moment problem 
with unknown support}\label{lmoment}

In the  generalized L-moment problem with $\K\subset\R^n$ known, the goal is to 
retrieve a 
measure $\phi$ on $\K$ with density $0\leq h\leq L$ in $\mathscr{L}^\infty(\K)_+$, whose moments $\boldsymbol{\phi}=(\phi_{d_x})_{d_x\in D}=(\int_\K \xx^{d_x} d\phi)_{d_x\in D}$, are given, and where $D\subset\N^n$ is also given.  The classical univariate L-moment problem on $\R$ was solved in \cite{L-moment}.

So when $\K$ is known and moments of Lebesgue measure on $\K$ are known as well, 
then we have seen how to apply the moment-completion approach to recover the unknown density
in $h\in\mathscr{L}(\K)_+$; see Section \ref{sec:main} with $\XX=\K$.

When $\K$ is unknown this is just the $\K$-moment problem with the additional
information that $\phi$ has a bounded density with respect to Lebesgue measure on 
the unknown $\K$. We next show how to formulate this problem as a completion problem investigated in previous sections.

Assuming $\K$ is compact, let $\XX$ be a box such that $\K\subset\XX$. Then
\begin{equation}
    \label{L-moment}
    \mu_{d_x} \,=\,\int_\K \xx^{d_x}\,d\phi\,=\,\int_\XX\xx^{d_x} \,\underbrace{1_\K(\xx)\,h(\xx)}_{f(\xx)}\,d\xx,\quad d_x\in D.
\end{equation}
Formulation \eqref{L-moment} for the  full L-moment problem (i.e., where $D=\N^n$) is exactly the framework developed in previous sections where 
$\xx\mapsto f(\xx):=1_\K(\xx)\,h(\xx)\in \mathscr{L}^\infty(\XX)_+$,
is the function whose graph is to be recovered. Indeed, letting
$d\mu(\xx,y)=\delta_{1_\K(\xx)\,h(\xx)}(dy)\,d\xx$,
\begin{eqnarray}
\nonumber
\mu_{d_x,0}&=&\int_\XX\xx^{d_x}\,d\xx\,,\quad d_x\in\N^n\quad \mbox{known in closed form}\\
\label{order-1}
\mu_{d_x,1}&=&\int_{\XX\times Y} \xx^{d_x}\,y\,d\mu\,:=\,\int_\XX \xx^{d_x}\,f(\xx)\,d\xx\,=\,\int_\K\xx^{d_x}\,h(\xx)\,d\xx\,=\,\phi_{d_x}\,,\:{d_x}\in\N^n.
\end{eqnarray}
So applying the moment completion approach developed earlier, one may 
approximate as closely as desired any finite number of moments of $\mu$,
and in a second-step to approximate the function $\xx\mapsto 1_\K(\xx)\,h(\xx)$.

It follows that the moment completion approach allows to both approximate 
the unknown support and density of the measure $\phi$ on $\K$.

\subsection*{The particular case when $h=1$}

When the density $h\in\mathscr{L}(\K)_+$ is the indicator
$1_\K$ then the above procedure described in Section \ref{sec:main} allows to 
approximate the indicator function $1_\K$ from knowledge of moments of the Lebesgue measure on $\K$ (e.g. obtained from measurements).
Notice  that in this case \eqref{order-1} yields:
\[\mu_{d_x,d_y}\,=\,\int_\XX \xx^{d_x}\,y^{d_y} d\mu(\xx,y)\,=\,
\int_\XX \xx^{d_x}1_\K(\xx)^{d_y}\,d\xx\,=\,\mu_{d_x,1}\,,\quad\forall (d_x,d_y)\in\N^{n+1}.\]
Therefore the matrix completion problem
is trivial. There is no optimization involved. Only the second step 
of recovery via the Christoffel-Darboux kernel is needed. 
In particular the moment matrix $\M_r(\mu)$ is highly singular since
all columns indexed by $(d_x,d_y)$ with $d_y>1$ are identical to the column indexed by $(d_x,1)$.
So we may replace $\M_r(\mu)$ with the smaller size matrix $\widehat{\M}_r(\mu)$ with rows and 
columns indexed by monomials $(\xx^{d_x})_{d_x\in\N^n_r}$ and $(\xx^{d_x}\,y)_{d_x\in\N^n_{r-1}}$.

\section{Graph approximation with the Christoffel-Darboux kernel}\label{cd-sec}

Let us briefly recall one of the main results of \cite{momgraph} where the Christoffel-Darboux kernel is used to approximate the support of a measure concentrated on a graph given its moments. 

Let $\mathbf{b}(\xx,y)$ be a polynomial vector of degree at most $d$ whose elements are orthonormal with respect to the bilinear form induced by a measure which is absolutely continuous with respect to the uniform measure on $\XX\times Y$. Let
\[
\M_r:=\int_{X} \mathbf{b}(\xx,f(\xx))\mathbf{b}(\xx,f(\xx))^\top d\xx = \int_{\XX\times Y} \mathbf{b}(\xx,y)\mathbf{b}(\xx,y)^\top d\mu(\xx,y)
\]
be the moment matrix of order $d$ of the measure $d\mu(\xx,y):=\mathbb{I}_X(\xx)\,d\xx\,\delta_{f(\xx)}(dy)$ concentrated on the graph
$\{(\xx,f(\xx)):\xx\in X\} \subset X\times Y$. Given $\M_r$, we want to compute an approximation $f_r$ of function $f$, with convergence guarantees when $r\to\infty$.

Given a regularization parameter $\beta_r := 2^{3 - \sqrt{r}}$, define the Christoffel-Darboux (CD) polynomial
\[
 q_r(\xx,y) := \mathbf{b}(\xx,y)^\top (\M_r + \beta_r \mathbf{I})^{-1} \mathbf{b}(\xx,y).
\]
Polynomial $q_r$ can be computed numerically by a spectral decomposition of the positive semidefinite matrix $\M_r$. Instead of trying to approximate the (possibly discontinuous) function $f$ with polynomials, we approximate it with a class of semi-algebraic functions. We define the semi-algebraic approximant
\[
f_r(\xx) := \min \{ \displaystyle\mathrm{argmin}_{y\in Y} \:\: q_r(\xx,y) \}
\]
as the minimum of the argument of the minimum of the CD polynomial, always well defined since this polynomial is a sum of squares (SOS).

\begin{theorem}\cite[Theorem 1]{momgraph}\label{l1approx}
If the set $S \subset X$ of continuity points of $f$ is such that $X \setminus S$ has Lebesgue measure zero, then $\lim_{r\to\infty}f_r(\xx)=f(\xx)$ for almost all $\xx \in X$ and $\lim_{r \to\infty}\Vert f-f_r\Vert_{{\mathscr L}^1(X)} = 0$.
\end{theorem}

This graph approximation method is implemented in the {\tt momgraph} function described in \cite{momgraph}.

\section{Examples}

For our numerical examples we modeled our moment problems with GloptiPoly \cite{gloptipoly} and solved them with MOSEK \cite{mosek} on Matlab. For reconstructing graphs from approximate moments we used he {\tt momgraph} function described in \cite{momgraph}.

\subsection{Convex optimal transport}

Let us first illustrate the CD kernel graph approximation method of Theorem \ref{l1approx} with a standard quadratic optimal transport problem, a particular case of optimization problem \eqref{opt}. Let $\xx=x\in \XX:=[-1,1]$ and $y \in \YY:=[-1,1]$ be scalar, and $c(x,y)=(x-y)^2/2$. Only marginal moments are available: $D = \{(d_x,0) : d_x \in \N\} \cup \{(0,d_y) : d_y \in \N\}$ and we want to transport the density $d\lambda(x)=(1-x)dx/2$ on $\XX$ onto the density $(1+y)dy/2$ on $\YY$. 

In this case there exists a unique transport map
\[
f(x)=-1+\sqrt{(1+x)(3-x)}
\]
yielding the optimal cost
\[
\int_{\XX} c(x,f(x))d\lambda(x) = \frac{\pi}{2}-\frac{4}{3} \approx 0.23746.
\]
\begin{table}[ht]
\centering
\begin{tabular}{c|ccccccc}
relaxation order $r$ & 1 & 2 & 3 & 4 & 5 & 6 & 7 \\\hline
lower bound & 0.22222 & 0.23494 & 0.23725 &  0.23742 & 0.23743 & 0.23745 & 0.23746
\end{tabular}
\caption{Lower bounds on optimal cost at various relaxation orders $r$.\label{copt-tab}}
\end{table}

\begin{figure}[ht]
\centering
\includegraphics[width=0.45\textwidth]{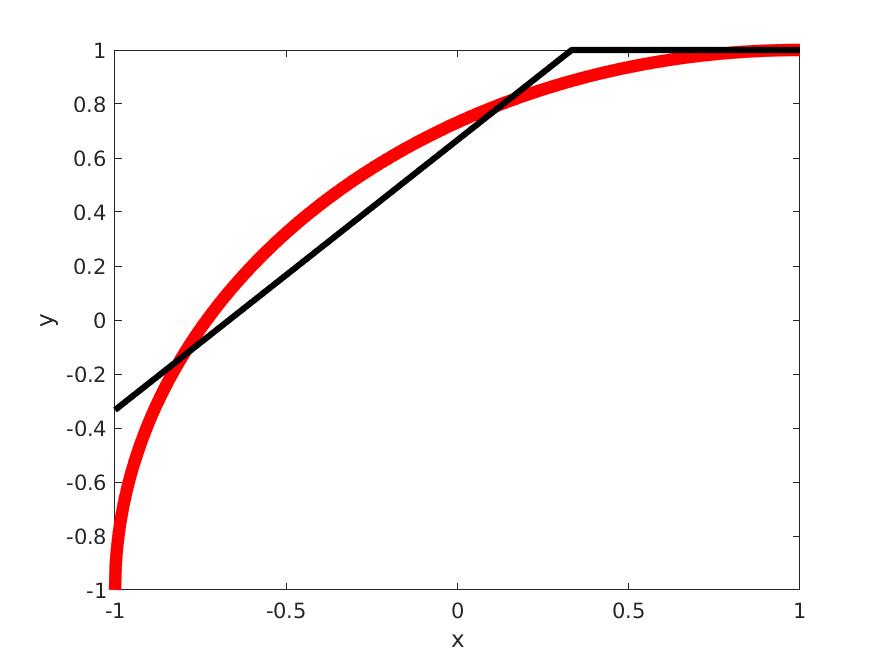}
\includegraphics[width=0.45\textwidth]{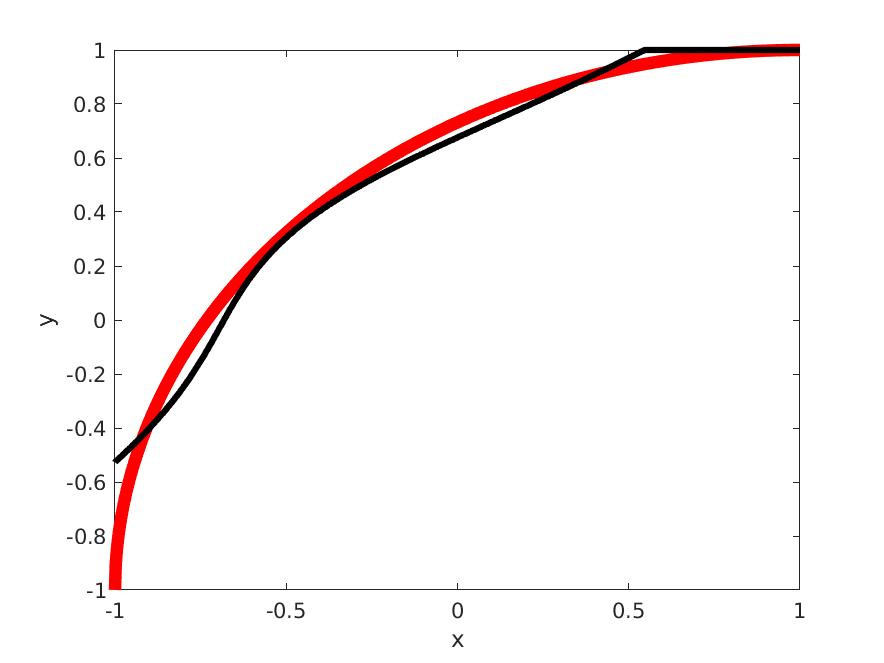}
\includegraphics[width=0.45\textwidth]{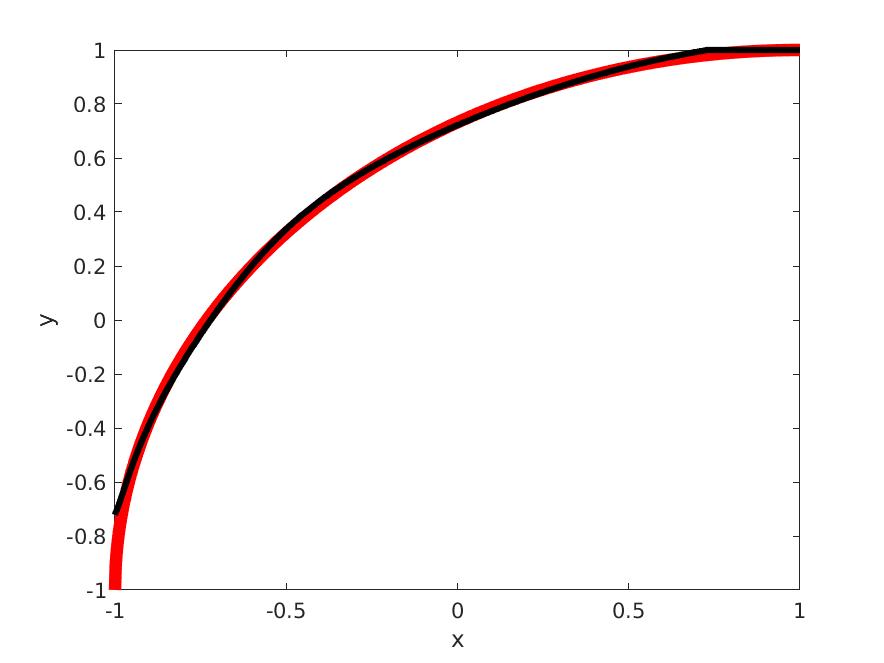}
\includegraphics[width=0.45\textwidth]{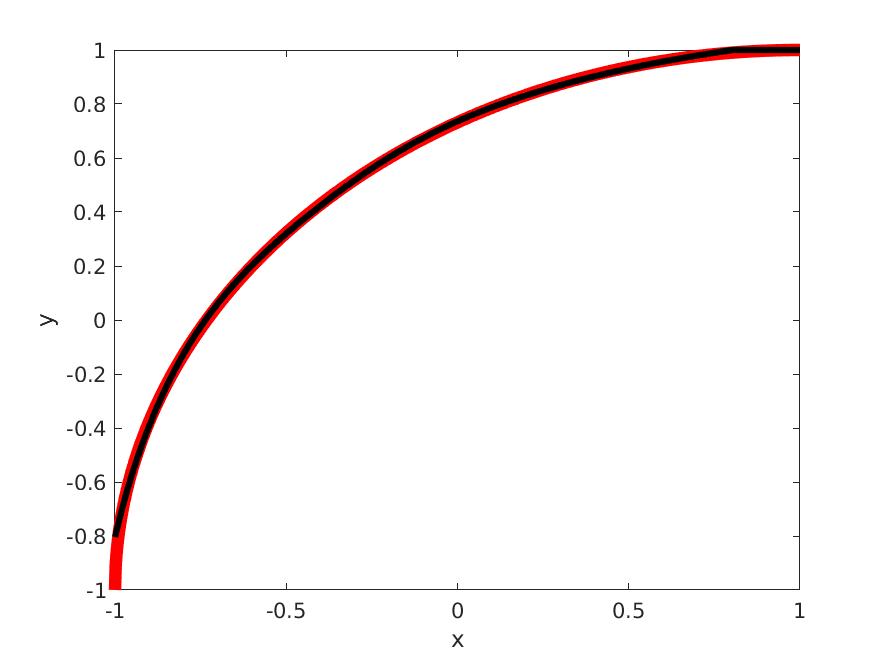}
\caption{CD approximations (thin black) of optimal transport map
(thick gray) for relaxation orders $r=1$ (top left), $r=2$ (top right), $r=3$ (bottom left), $r=4$ (bottom right).\label{copt-fig}}
\end{figure}

In Table \ref{copt-tab} we report the lower bounds of the moment 
relaxations for various relaxation orders $r$. In Figure \ref{copt-fig} we represent the approximate graphs obtained from the CD kernel method described in Section \ref{cd-sec}.

\subsection{Non-convex optimal transport}

Let us now illustrate the CD kernel graph approximation method of Theorem \ref{l1approx} with a non-convex optimal transport problem. Let $\xx=x \in \XX:=[0,1]$ and $y \in \YY:=[0,1]$ be scalar, and $c(x,y):=(4x-y)xy$. Only marginal moments are available: $D = \{(d_x,0) : d_x \in \N\} \cup \{(0,d_y) : d_y \in \N\}$ and we want to transport the Lebesgue measure on $X$ onto the Lebesgue measure on $Y$.  An analytic optimal transport map is known \cite[Section 5.3]{anderson87} for this example:
\[
f(x) = \left\{
\begin{array}{ll}
\frac{2}{3}+\frac{4}{3}x & \text{on $[0,\frac{1}{4}]$} \\
2-4x & \text{on $[\frac{1}{4},\frac{1}{3}]$} \\
1-x & \text{on $[\frac{1}{3},1]$} \\
\end{array}\right.
\]
yielding the optimal cost
\[
\int_X c(x,f(x))dx = \frac{107}{432} \approx 0.24769.
\]
Formulating problem \eqref{opt} as a generalized problem of moments \eqref{relax} with $\lambda(dx)=dx$, we can use the moment-SOS hierarchy to approximate numerically all the moments of the measure $dx\delta_{f(x)}(dy)$. 

 \begin{figure}
        \centering
        \includegraphics[width=\textwidth]{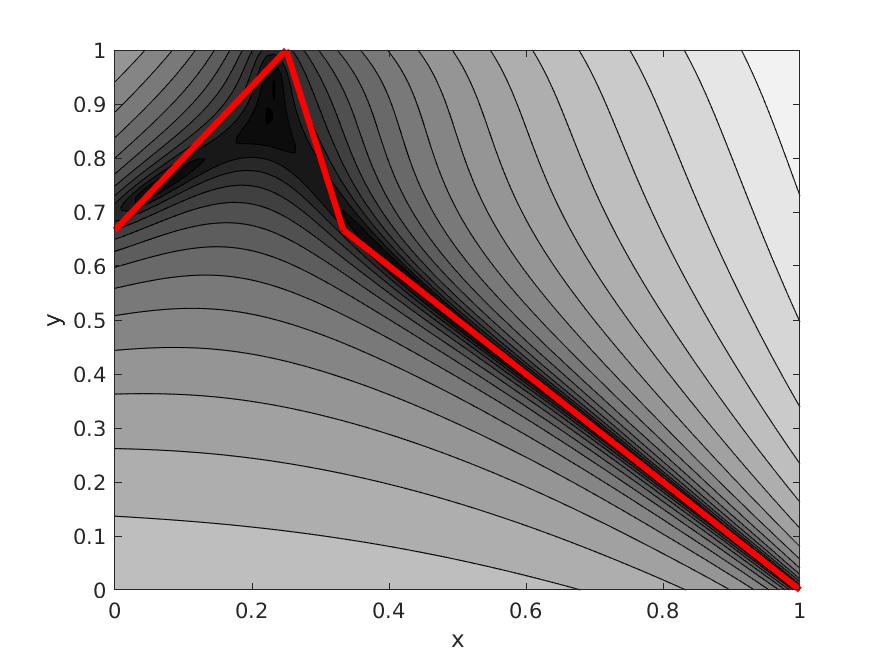}
        \caption{Logarithmic level sets (gray, smaller values darker) of the CD polynomial of degree $12$ approximating the graph (thick red) of the optimal transport map.}
        \label{fig-transport}
    \end{figure}
    
We solved the moment relaxation of order $r=6$, corresponding to approximate moments of degree up to $2r=12$ and a positive semidefinite moment matrix of size $28$. The lower bound on the optimal cost obtained with the moment relaxation is $0.24741$, matching 3 digits of the exact cost. The maximum absolute error between the approximate moments and the exact moments is $9.6\cdot 10^{-4}$. On Figure \ref{fig-transport} we represent sublevel sets of the CD polynomial $q_6$ of degree $12$. We observe that they closely match with the graph of the transport map.

\subsection{Graph reconstruction from linear measurements}

 \begin{figure}
        \centering
        \includegraphics[width=.45\textwidth]{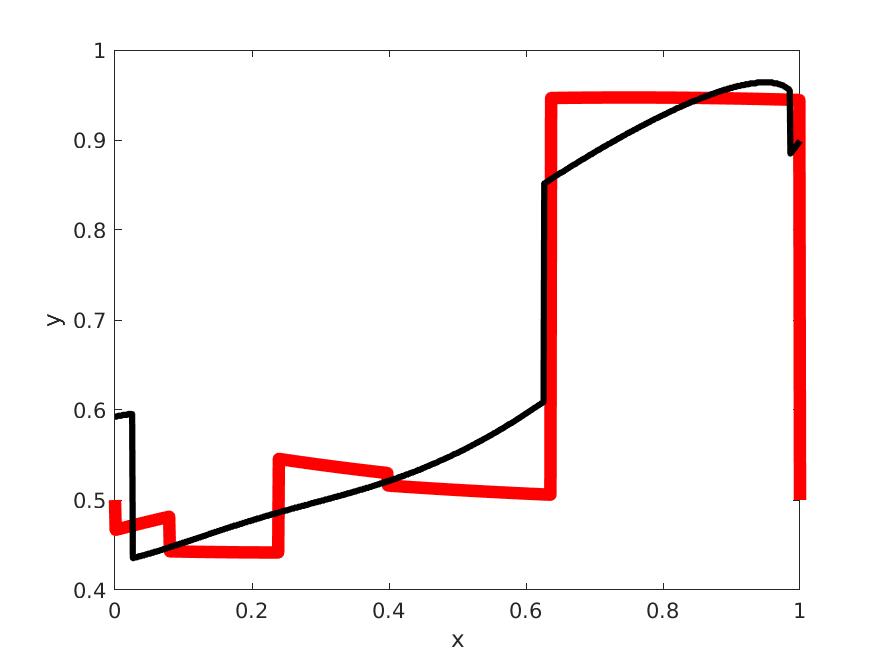}\includegraphics[width=.45\textwidth]{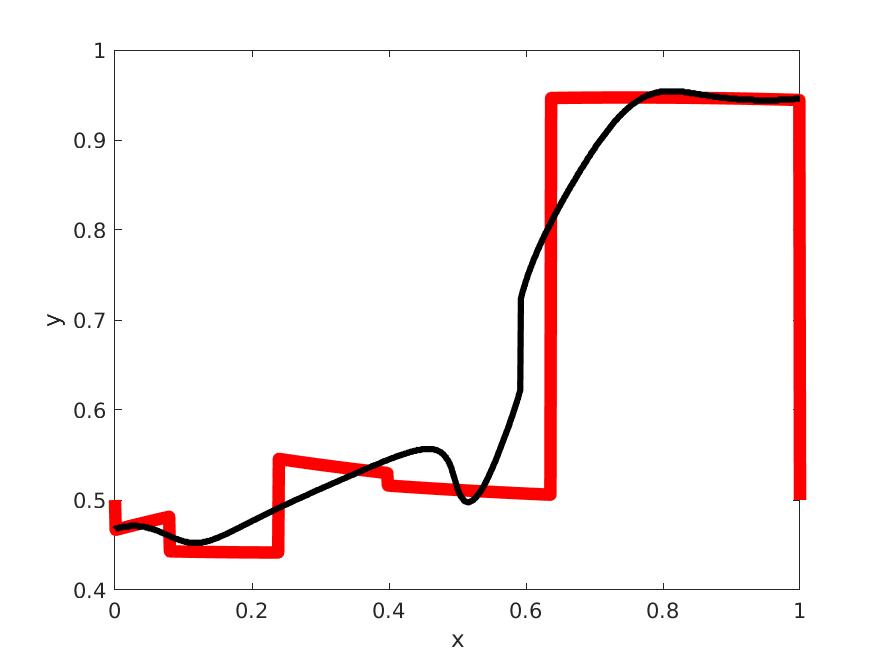}
        \includegraphics[width=.45\textwidth]{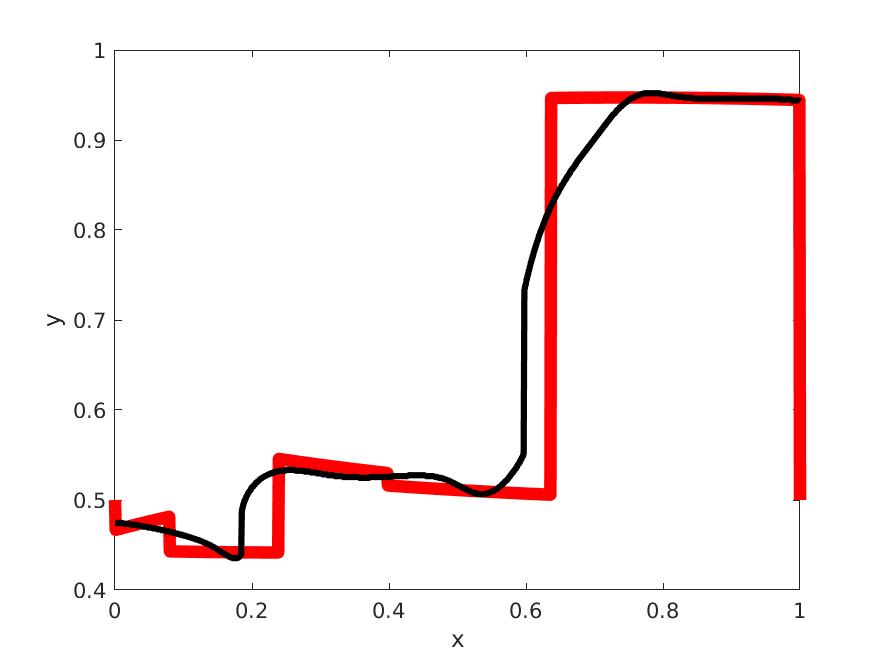}\includegraphics[width=.45\textwidth]{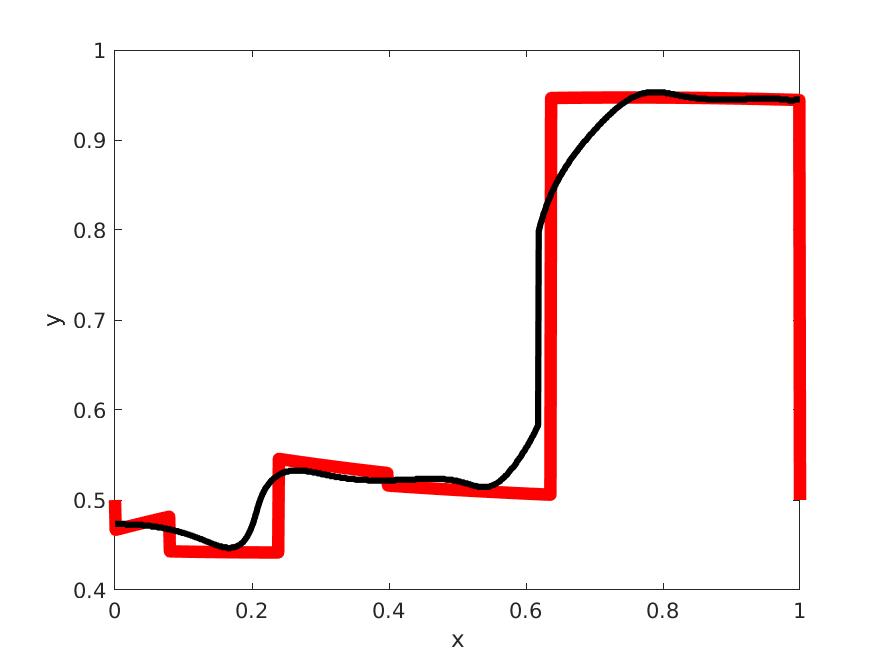}
        \caption{Discontinuous benchmark function (red) and its approximations (black) obtained from approximate moments recovered at relaxation order $r=4$ (top left), $r=8$ (top right), $r=12$ (bottom left), $r=16$ (bottom right).}
        \label{fig-eckhoff3}
    \end{figure}

Example 67 of \cite{eckhoff93} was used in \cite[Section 5.4]{momgraph} to illustrate the CD kernel graph approximation algorithm when all the moments are available. Let us revisit the same example to illustrate the reconstruction algorithm of Theorem 
\ref{th-infinite}, when  $\boldsymbol{\theta}$ is the truncated sequence of all ones.

Approximate graphs for various relaxation orders are reported in Figure \ref{fig-eckhoff3}.

\subsection{Comparison with $L_2$ estimate}

\begin{figure}
	\centering
	\includegraphics[width=0.45\textwidth]{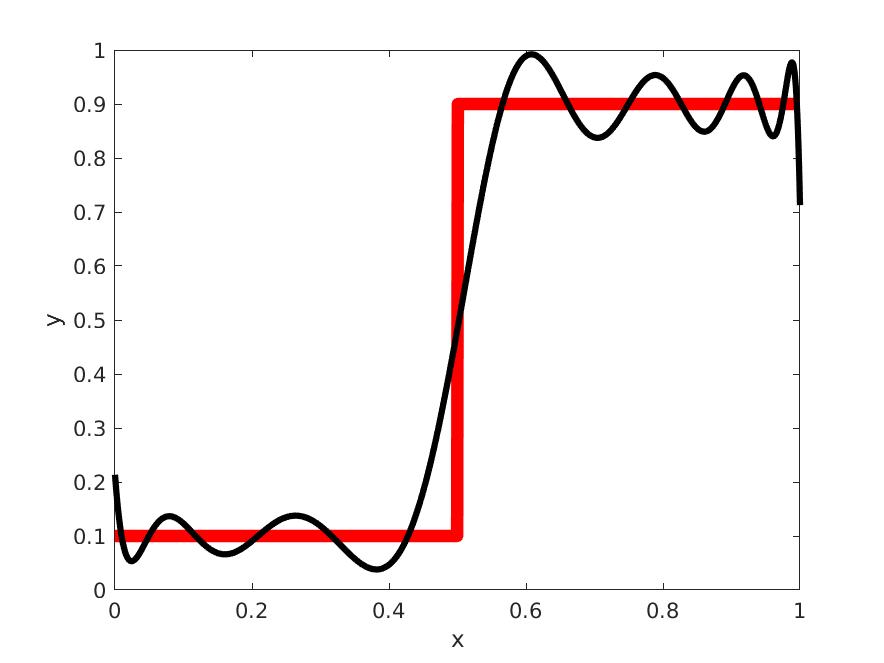}
	\includegraphics[width=0.45\textwidth]{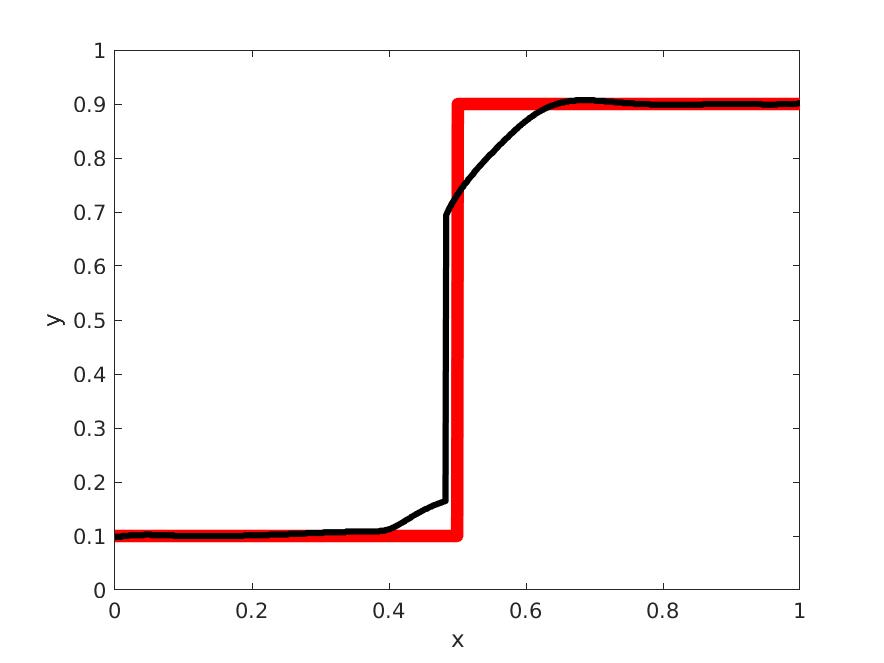}
	\caption{Approximations of the graph (thick red) of a step function.
	Left: $L_2$ estimate of degree $12$ (black).
	Right: CD polynomial of degree $12$ obtained from first order moments (black). 
	\label{fig-step}}
\end{figure}

\begin{figure}
	\centering
	\includegraphics[width=0.45\textwidth]{step12L2.jpg}
	\includegraphics[width=0.45\textwidth]{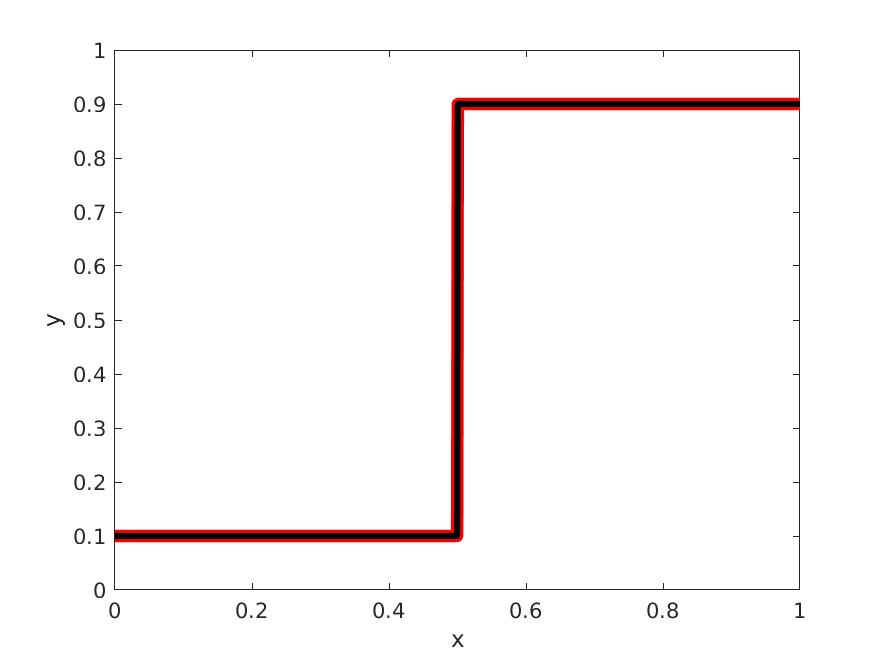}
	\caption{Approximations of the graph (thick red) of a step function.
	Left: $L_2$ estimate of degree $12$ (black).
	Right: CD polynomial obtained from all moments of degree up to $4$ (black). 
	\label{fig-stepcd}}
\end{figure}

Consider the problem of reconstructing a step function
\[
f(x) = \left\{\begin{array}{rcl}
0.1 & \mathrm{if} & 0\leq x<0.5 \\
0.9 & \mathrm{if} & 0.5\leq x\leq 1.
\end{array}\right.
\]
The only input data are the moments $\int_0^1 x^{d_x} dx$ and $\int_0^1 x^{d_x} f(x) dx$ for $d_x$ less than a given degree $r \in \N$.

The reconstruction algorithm of Theorem 
\ref{th-infinite}, when  $\boldsymbol{\theta}$ is the truncated sequence of all ones, and $r=12$ gives the approximate graph of the right of Figure \ref{fig-step}. It is not perfect, especially at the discontinuity point, yet it compares favorably with the $L_2$ density estimate of degree $12$ of \cite[Proposition 1]{hlm14} which oscillates around the graph to be recovered, as seen at the left of Figure \ref{fig-step}.

For comparison, we report on the right of Figure \ref{fig-stepcd} the graph obtained with the CD polynomial constructed from moments up to degree $4$, as already studied in \cite[Example 3, Figure 2]{momgraph}. We observe that in this case the graph approximation is almost perfect, thanks to the additional information provided by the higher degree moments.

\subsection{Disconnected L-moment problem}

 \begin{figure}
        \centering
        \includegraphics[width=.8\textwidth]{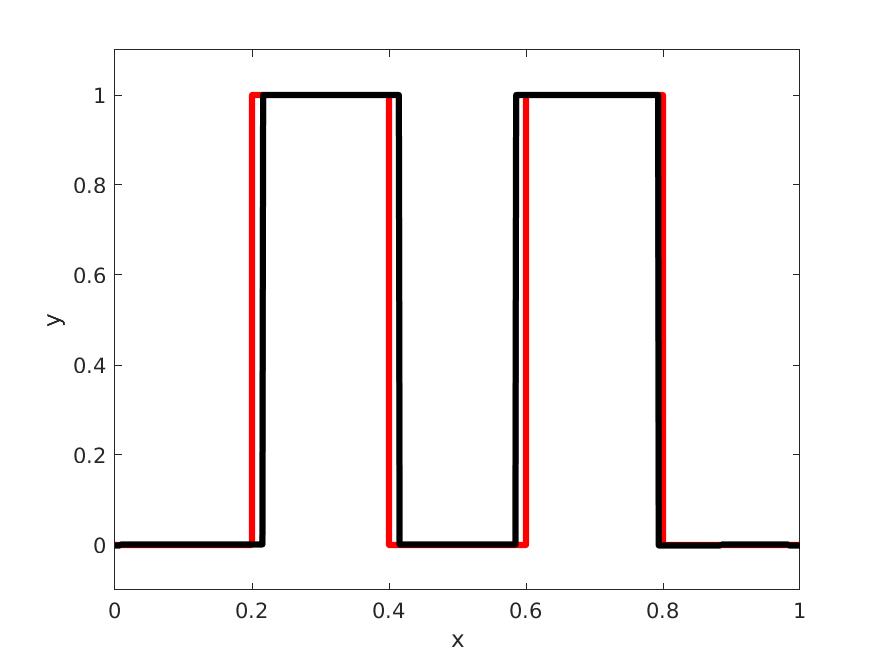}
        \caption{Indicator function (red) and its approximation (black) obtained from approximate moments recovered at relaxation order $r=10$.}
        \label{fig-lmomentgr}
    \end{figure}

\begin{figure}
        \centering
        \includegraphics[width=.45\textwidth]{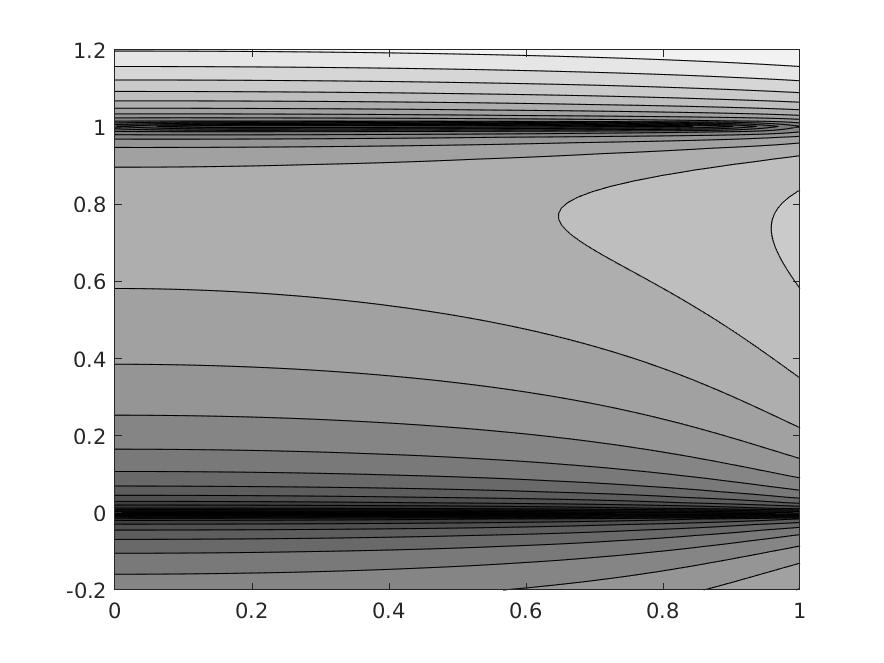} \includegraphics[width=.45\textwidth]{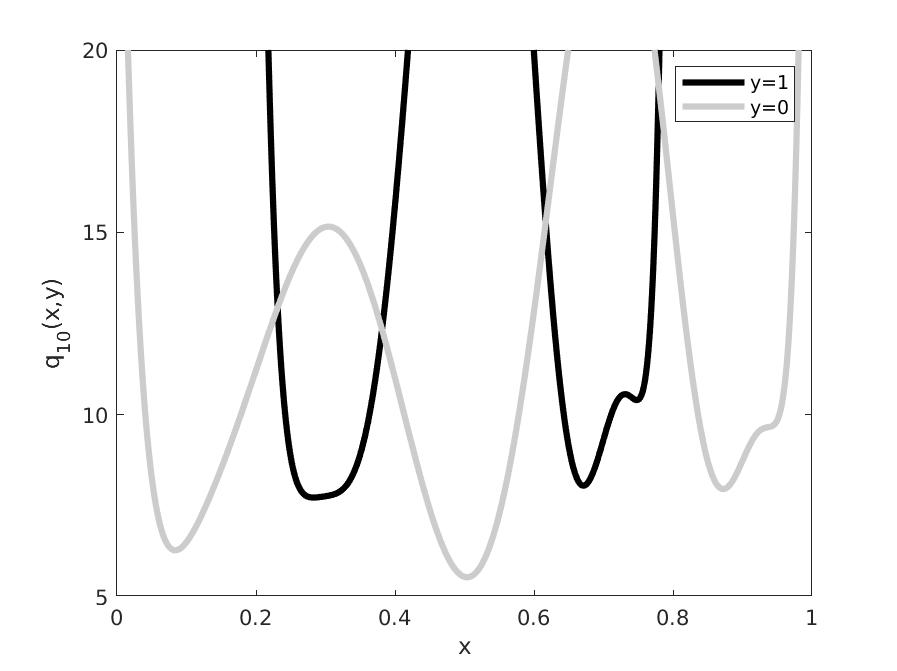}
        \caption{Logarithmic level sets (gray, smaller values darker) of the CD polynomial $q_{10}(x,y)$ of degree $20$ (left), and its values at $y=0$ (light gray) and $y=1$ (black).}
        \label{fig-lmomentcd}
    \end{figure}

This example illustrates our solution in Section \ref{lmoment} to the L-moment problem when $h=1$ and hence the graph to be reconstructed is the indicator function of an unknown set $\K$. To compute the moments we choose a disconnected set $\K:=[1/5,2/5]\cup[3/5,4/5] \subset \XX:=[0,1]$. Solving the moment relaxation of order $r=10$ we obtain the approximate graph represented on Figure \ref{fig-lmomentgr}. We observe that it matches closely the indicator function of $\K$.

The level sets of the CD polynomial (obtained with the {\tt momgraph} function with regularization parameter $10^{-6}$) are represented on the left of Figure \ref{fig-lmomentcd}. We see that the CD polynomial $q_{10}(x,y)$ is small around $y=0$ and $y=1$, as expected. For each $x$, the value of $y$ for which $q_{10}(x,y)$ is minimal is the approximation $f_{10}(x)$ of the indicator function of $\K$. On the right of Figure \ref{fig-lmomentcd} we represent $q_{10}(x,y)$ for $y=0$ and $y=1$. For each $x$, the value $f_{10}(x)$ is the minimum of $q_{10}(x,0)$ and $q_{10}(x,1)$.

\section{Proofs}\label{proofs}
\subsection*{Proof of Theorem \ref{th1}}

As $\phi_{d_x,0}=\lambda_{d_x}$ for all $d_x\in\N^n$
and $\XX\subset [0,1]^n$, the marginal of $\phi$ on 
$\XX$ is the measure $\lambda$. Next,
disintegrate $\phi$ with respect to its marginal $\lambda$ on $\XX$ and its conditional $\varphi(\cdot\vert \xx)$ on $\YY$ given $\xx\in\XX$, i.e.:
\[\phi(A\times B)\,=\,\int_A\varphi(B\vert \xx)\,\lambda(d\xx),\qquad
A\in\mathscr{B}(\XX),\,B\in\mathscr{B}(\YY).\]
In particular,
\[\phi_{d_x,1}\,=\,\int_{\XX\times\YY} \xx^{d_x}y\,d\mu(\xx,y)\,=\,
\int_{\XX} \xx^{d_x}\,\left(\underbrace{\int_{\YY}y\,\varphi(dy\vert\xx)}_{h_1(\xx)}\right)\,d\lambda(\xx)\,=\,\int_{\XX} \xx^{d_x}\,h_1(\xx)\,d\lambda(\xx)
\]
for some measurable function $h_1$ on $\XX$. Therefore 
\[\int_{\XX} \xx^{d_x}\,f(\xx)\,d\lambda(\xx)\,=\,\mu_{d_x,1}\,=\,
\phi_{d_x}\,=\,\int_{\XX}\xx^{d_x}\,h_1(\xx)\,d\lambda(\xx),\quad\forall d_x\in\N^n.\]
As signed measures on compact sets are moment-determinate, one deduces that
$h_1=f$ for $\lambda$-almost-all $\xx$ in $\XX$. Next,
fix $(d_x,d_y)\in\N^{n+1}$, arbitrary and recall that $\XX\subset[0,1]^n\subset\R^n_+$. Then
\begin{eqnarray*}
\phi_{d_x,d_y}\,=\,\int \xx^{d_x}y^{d_y}\,d\phi(\xx,y)&=&
\int_{\XX} \xx^{d_x}\,\left(\int_{\YY} y^{d_y}\varphi(dy\vert \xx)\right)\,d\lambda(\xx)\\
&\geq&\int_{\XX} \xx^{d_x}\,\left(\int_{\YY}y\varphi(dy\vert \xx)\right)^{d_y}\,d\lambda(\xx)\quad\mbox{[by Jensen's inequality}]\\
&\geq&\int_{\XX} \xx^{d_x}\,h_1(\xx)^{d_y}\,d\lambda(\xx)
\,=\,\int_{\XX} \xx^{d_x}\,f(\xx)^{d_y}\,d\lambda(\xx)\,=\,\mu_{d_x,d_y}.
\end{eqnarray*}
(We can apply Jensen's inequality because $y^{d_y}$ is convex on $\YY\subset\R_+$ and $\XX\subset \R^n_+$.)
Therefore with $r\in\N$ fixed, arbitrary, summing up 
over all $(2d_x,2d_y)\in\N^{n+1}_{2r}$ yields 
\[{\rm trace}\:\M_r(\boldsymbol{\phi})\,=\,\displaystyle\sum_{(d_x,d_y)\in\N^{n+1}_r}\phi_{2d_x,2d_y}
\,\geq\,\displaystyle\sum_{(d_x,d_y)\in\N^{n+1}_r}\mu_{2d_x,2d_y}\,=\,{\rm trace}\:\M_r(\boldsymbol{\mu}),\]
the desired result. 

(Concerning Remark \ref{positive-0}. If $\XX\not\subset\R^n_+$ the above arguments still hold whenever 
$d_x$ is even (hence $\xx^{d_x}\geq0$) so that Jensen's inequality applies.) $\Box$

\subsection*{Proof of Theorem \ref{th-finite}}
(i) Let $\boldsymbol{\phi}=(\phi_{d_x,d_y})_{(d_x,d_y)\in\N^{n+1}_{2r}}$ be an arbitrary feasible solution.
From $\M_r(g_j\,\boldsymbol{\phi})\succeq0$, we may deduce
\begin{equation}    \label{bounds}
L_{\boldsymbol{\phi}}(y^{2r})\,\leq\,\gamma^{2r}\,;\quad L_{\boldsymbol{\phi}}(x_i^{2r})\,\leq\,1\,\quad i=1,\ldots,n\,.\end{equation}
Hence by \cite[Proposition 2.38, p. 41]{l15}, $\vert\phi_{d_x,d_y}\vert\leq 
\max[1,\gamma^{2r}]=:\rho_r$ for all $(d_x,d_y)\in\N^{n+1}_{2r}$, which implies that that feasible set is compact and \eqref{relax-primal} has an optimal solution $\boldsymbol{\phi}^r$ with $\mathscr{T}_s(\boldsymbol{\phi}^r)=\tau_r$.

(ii) Let $(\boldsymbol{\phi}^r)_{r\in\N}$ be a sequence of arbitrary optimal solutions of \eqref{relax-primal} as $r$ increases. 
With each $r$ fixed, recalling 
the definition of 
$\rho_r$ in the proof of (i), one has $\phi^r_{0,0}=1$ and:
\[\vert\phi^r_d\vert\,\leq\,\rho_t,\quad\forall 2t-1\, \leq\,\vert d\vert\,\leq 2t;\quad t=1,\ldots,r.\]
So for each $r\in\N$ let 
\[\vert\psi^r_d\vert\,:=\,\frac{\vert\phi^r_d\vert}
{\rho_t},\quad\mbox{for all $2t-1\leq \vert d\vert\leq 2t;\:t=1,\ldots,r$}.\]
By construction $\vert\psi^r_d\vert<1$ for all $d\in\N^{n+1}_{2r}$.
Hence by completing the finite sequence $\boldsymbol{\psi}^r$ with zeros to make it an infinite sequence indexed by $\N^{n+1}$, the infinite sequence $\boldsymbol{\psi}^r=(\psi^r_d)_{d\in\N^{n+1}}$ is an element of the unit ball $\mathbf{B}_1$ of the Banach space $\ell_\infty$ of uniformly bounded sequence. As this unit ball is weak-star sequentially compact, there exists an element
$\boldsymbol{\psi}^*\in\mathbf{B}_1$ and a subsequence $(r_j)_{j\in\N}$ such that $\boldsymbol{\psi}^{r_j}\stackrel{weak~\star}{\to}\boldsymbol{\psi}^*$ as $j\to\infty$. In particular,
\[\lim_{j\to\infty}\psi^{r_j}_d\,=\,\psi^*_d,\quad \forall d\in\N^{n+1},\]
and equivalently:
\begin{equation}
    \label{convergence}
    \lim_{j\to\infty}\phi^{r_j}_{k,\alpha}\,=\,
\rho_t\,\psi^*_d\,=:\,\phi^*_d,\quad \forall\, 2t-1\,\leq \vert d\vert\,\leq 2t\,;\:t=1,2,\ldots\end{equation}
In addition, the latter convergence  \eqref{convergence} implies
$\phi^*_d\geq0$ for all $d\in\N^{n+1}$, and 
\[\M_r(g_j\,\boldsymbol{\phi}^*)\,\succeq0,\quad j=0,1,\ldots,m+2\,,\:\forall 
r\in\N.\]
Next, recall that the quadratic module $Q(g)$ in \eqref{Q(g)}
is Archimedean and therefore 
by Putinar's Positivstellensatz \cite{putinar93}, $\boldsymbol{\phi}^*$ has a representing measure $\phi^*$ on $\XX\times\YY$.
Moreover, by \eqref{convergence} again,
\[\phi^*_{d_x,0}\,=\,\lambda_{d_x}\quad\mbox{and}\quad\phi^*_{d_x,1}\,=\,\mu_{d_x,1},\qquad\forall d_x\in\N^n,\]
and $\lim_{j\to\infty}\tau_{r_j}=\lim_{j\to\infty}
\mathscr{T}_s(\boldsymbol{\phi}^{r_j})=
\mathscr{T}_s(\boldsymbol{\phi}^*)$. Therefore since
$\tau_{r_j}\leq \mathscr{T}_s(\boldsymbol{\mu})$ for all $j$, we obtain
$\mathscr{T}_s(\boldsymbol{\phi}^*)\leq \mathscr{T}_s(\boldsymbol{\mu})$.
But by Lemma \ref{th1}, one concludes that
$\mathscr{T}_s(\boldsymbol{\phi}^*)=
\mathscr{T}_s(\boldsymbol{\mu})$, and
$\phi^*_d=\mu_d$ for all $d\in\N^{n+1}_{2s}$, i.e.,
\[\lim_{j\to\infty}\phi^{r_j}_d\,=\,
\mu_d\,=\,\int_{\boldsymbol{\Omega}}\xx^{d_x}\,f(\xx)^{d_y}\,d\lambda(\xx),\quad\forall d=(d_x,d_y)\in\N^{n+1}_{2s}.\]
As this limit does not depend on the particular converging subsequence $(\boldsymbol{\phi}^{r_j})_{j\in\N}$, we also obtain
the full convergence \eqref{th-finite-completion}.
$\Box$

\subsection*{Proof of Theorem \ref{th-infinite}}

The proof of (i) is almost a verbatim copy of that of Theorem \ref{th-finite}(i). 

(ii) Again as in the proof of Theorem \ref{th-finite}(ii) there is a subsequence $(r_j)_{j\in\N}$ and a sequence $\boldsymbol{\phi}^*$ such that \eqref{convergence} holds and $\boldsymbol{\phi}^*$ has a representing measure on $\XX\times\YY$.
In addition, recall that \eqref{bounds} holds for $\boldsymbol{\phi}^{r_j}$, and so
$\phi_d^{r_j}\leq \gamma^{2\lceil d_y/2\rceil}$ for all $d=(d_x,d_y)\in\N^{n+1}$. 

Hence letting
$\boldsymbol{\psi}:=(\psi_d)_{d\in\N^{n+1}}$  with $\psi_{d}=\gamma^{2\lceil d_y/2\rceil}$,
$\boldsymbol{\phi}^{r_j}\leq\boldsymbol{\psi}$ for all $j$. Moreover notice that $\mathscr{T}(\boldsymbol{\psi})<\infty$. 
Therefore by the Dominated Convergence Theorem \cite[p. 49]{ash}, $\mathscr{T}(\boldsymbol{\phi}^{r_j})\to \mathscr{T}(\boldsymbol{\phi}^*)$, as $j\to\infty$.
But recall that $\mathscr{T}(\boldsymbol{\phi}^r)\leq \mathscr{T}(\boldsymbol{\mu})$ for all $r$, and therefore $\mathscr{T}(\boldsymbol{\phi}^*)\leq \mathscr{T}(\boldsymbol{\mu})$. 

On the other hand, by Theorem \ref{th1}, the reverse inequality $\mathscr{T}(\boldsymbol{\phi}^*)\geq \mathscr{T}(\boldsymbol{\mu})$ holds (since $\phi^*_d\geq\mu_d$ for all $d$), which
yields $\mathscr{T}(\boldsymbol{\phi}^*)=\mathscr{T}(\boldsymbol{\mu})$. But from $\mathscr{T}(\boldsymbol{\phi}^*-\boldsymbol{\mu})=0$ we obtain
\[0\,=\,\sum_{d\in\N^{n+1}}\theta_d\,(\underbrace{\phi^*_d-\mu_d}_{\geq0})\,\Rightarrow\:
\phi^*_d=\mu_d,\quad\forall d\in\N^{n+1}.\]
Finally, as the limit $\phi^*_d=\mu_d$ does not depend on the particular converging subsequence $(r_j)_{j\in\N}$,
we conclude that the whole sequence converges, i.e., the desired result \eqref{th-infinite-completion} holds.  $\Box$

\section{Conclusion}

This paper deals with the problem of recovering the graph of a function when only partial moment information is available, here only zero-th and first-order moments, namely linear measurements of the function. We show 
how to recover the graph asymptotically by solving 
a hierarchy of semidefinite programs, with an appropriately chosen objective function that promotes sparsity. 

As shown in non trivial numerical experiments, our results 
suggest a perhaps counter-intuitive message. Namely we observe experimentally that
it is better 

- to approximate the graph $G=\{(\xx,f(\xx)) :\xx\in \XX\} \subset \XX\times\YY$ from moments of the measure supported on $G$ (a thin subset of $\XX\times\YY$) with the recovery technique of \cite{momgraph}, 

rather than 

- to approximate the function $f$ itself (e.g. via $L^2$-norm approximation by polynomials) from moments of the measure with density $f$ on $X$.

Moreover,
this is true even if \emph{all} exact moments of the measure with density $f$ are available, whereas only a few  exact moments of the measure supported on the graph $G$ are available (and the other 
must be appropriately approximated).

\section*{Acknowledgments}

The authors wish to thank Swann Marx and Edouard Pauwels for fruitful and helpful discussions on the topic.

\end{document}